\let\chooseClass1   
\let\chooseClass3   

\ifx\chooseClass1
\documentclass[12pt,a4paper,reqno]{amsart}
\usepackage[cp1251]{inputenc}
\usepackage[T2A]{fontenc}

\newlength{\Totalheight}
\fi

\ifx\chooseClass2
\IfFileExists{extarticle.cls}{
\documentclass[14pt]{extarticle}
\emergencystretch 7 pt
\setlength{\oddsidemargin}{-17 mm} 
\setlength{\textwidth}{192 mm}     
\setlength{\textheight}{246 mm}    
\setlength{\topmargin}{-31 mm}     
}{
\documentclass[12pt]{article}
\textheight = 8.5in
\textwidth 6.3in
\setlength{\oddsidemargin}{0mm}
\setlength{\topmargin}{-15 mm}
}{}
\fi

\ifx\chooseClass3
\documentclass[12pt]{article}
\fi

\ifx\chooseClass1
	\else
\makeatletter
\def\@seccntformat#1{\csname the#1\endcsname.\quad}
\renewcommand\section{\@startsection {section}{1}{\z@}%
                                   {-3.5ex \@plus -1ex \@minus -.2ex}%
                                   {2.3ex \@plus.2ex}%
                                   {\normalfont\large\bfseries}}
\renewcommand\subsection{\@startsection{subsection}{2}{\z@}%
                        {3.25ex plus 1ex minus .2ex}{-.5em}%
                        {\normalfont\normalsize\bfseries}}
\renewcommand\subsubsection{\@startsection{subsubsection}{3}{\z@}%
                        {3.25ex plus 1ex minus .2ex}{-.5em}%
                        {\normalfont\normalsize\it}}
\@addtoreset{equation}{section}
\makeatother
\fi

\usepackage{ifpdf}
\ifpdf
    \IfFileExists{hyperref.sty}{\usepackage[pdftex]{hyperref}}{}
\else
    \IfFileExists{hyperref.sty}{\usepackage[hypertex]{hyperref}}{}
\fi

\usepackage{amsthm}
\newtheoremstyle{boldhead}
{\topsep}
{\topsep}
{\slshape}
{}
{\bfseries}
{.}
{ }
{\thmname{#1}\thmnumber{ #2}\thmnote{ (#3)}}

\newtheoremstyle{boldremark}
{\topsep}
{\topsep}
{\upshape}
{}
{\bfseries}
{.}
{ }
{\thmname{#1}\thmnumber{ #2}\thmnote{ (#3)}}

\swapnumbers

\theoremstyle{boldhead}
\newtheorem{theorem}[subsection]{Theorem}
\newtheorem{corollary}[subsection]{Corollary}
\newtheorem{lemma}[subsection]{Lemma}
\newtheorem{proposition}[subsection]{Proposition}

\theoremstyle{boldremark}

\newtheorem{example}[subsection]{Example}

\newtheorem{remark}[subsection]{Remark}

\usepackage{    amsmath,
                amsfonts,
                amssymb,
}

\usepackage[matha,mathx]{mathabx}

\numberwithin{equation}{section}

\IfFileExists{url.sty}{\usepackage{url}}{}
\providecommand{\url}[1]{{\tt #1}}

\IfFileExists{euscript.sty}{\usepackage[mathcal]{euscript}}{}

\usepackage{diagrams}
\diagramstyle[height=2em,balance,righteqno,PostScript=dvips,nohug]

\newarrowhead{pto}{\mkern-6mu\raise.162em\hbox{$\sss\succ$}\mkern6mu}%
{\mkern6mu\raise.162em\hbox{$\sss\prec$}\mkern-6mu}%
{\sss\curlyvee}{\sss\curlywedge}

\newarrowtail{mfrom}{\mkern6mu\lower.018em\hbox{$\sss\prec$}\mkern-6mu}%
{\mkern-6mu\lower.018em\hbox{$\sss\succ$}\mkern6mu}%
{\sss\curlywedge}{\sss\curlyvee}

\newarrowtail{sscurlyvee}{\raise.18ex\hbox{$\sss\succ$}}%
{\mkern-.5mu\raise.18ex\hbox{$\sss\prec$}\mkern.4mu}%
{\sss\curlyvee}{\sss\curlywedge}

\newarrow{Cof}{sscurlyvee}---{->}
\newarrow{DashTo}{}{dash}{}{dash}{->}
\newarrow{Epi}----{triangle}
\newarrow{Fib}----{->>}
\newarrow{FromTo}{mfrom}==={pto}
\newarrow{Gyrlyga}{boldhook}----
\newarrow{MapsTo}{mapsto}---{->}
\newarrow{Mono}{boldhook}---{->}
\newarrow{TTo}----{->}
\newarrow{Twoar}===={=>}

\newcommand\NN{{\mathbb N}}
\newcommand\ZZ{{\mathbb Z}}

\newcommand{\cc}{{\mathcal C}}

\newcommand{\cw}{{\mathcal W}}

\newcommand{\bull}{{\scriptscriptstyle\bullet}}
\newcommand{\bulle}{{\scriptstyle\bullet}}
\newcommand{\colim}{\qopname\relax m{colim}}
\newcommand{\Com}{{{\mathsf C}_\kk}}

\newcommand{\sS}[2]{\vphantom{#2}#1 #2}
\newcommand{\n}[1]{\nobreakdash-\hspace{0pt}}
\newcommand{\ainf}[1]{$A_\infty$\nobreakdash-\hspace{0pt}}
\newcommand{\ainfm}[1]{$\mathrm{A}_\infty$\nobreakdash-\hspace{0pt}}

\let\eps\varepsilon
\let\ge\geqslant
\let\kk\Bbbk

\let\rto\xrightarrow
\let\sss\scriptstyle
\let\tens\otimes
\let\ttt\textstyle

\newcommand\KK{\mathbb{K}}

\newcommand\cL{{\mathcal L}}

\newcommand\cR{{\mathcal R}}

\newcommand\cW{{\mathcal W}}

\newcommand{\modul}{\textup{-mod}}

\DeclareMathOperator\Cone{Cone}

\DeclareMathOperator\dg{\mathbf{dg}}
\DeclareMathOperator\udg{\underline{\dg}}

\DeclareMathOperator\gr{\mathbf{gr}}

\DeclareMathOperator\id{id}

\DeclareMathOperator\inj{in}

\DeclareMathOperator\Mor{Mor}

\DeclareMathOperator\Ob{Ob}

\DeclareMathOperator\pr{pr}

\DeclareMathOperator\Set{\mathcal Set}

\newcommand{\exaref}[1]{Example~\ref{#1}}

\newcommand{\lemref}[1]{Lemma~\ref{#1}}
\newcommand{\propref}[1]{Proposition~\ref{#1}}
\newcommand{\remref}[1]{Remark~\ref{#1}}

\newcommand{\thmref}[1]{Theorem~\ref{#1}}

\newlength{\texthigh}
\begin{document}
\bibliographystyle{amsalpha}

\ifx\chooseClass1
УДК 512.58\\[3mm]

{\bf\noindent
МОДЕЛЬНА СТРУКТУРА НА КАТЕГОРІЯХ\\[1mm]
ПОВ’ЯЗАНИХ З КАТЕГОРІЯМИ КОМПЛЕКСІВ
}

\bigskip

\title[A model structure on categories related to categories of complexes]{A model structure on categories\\
related to categories of complexes}
	\else
\title{A model structure on categories\\
related to categories of complexes}
\fi

\ifx\chooseClass1
\author{V. Lyubashenko}
	\else
\author{Volodymyr Lyubashenko}
\date{Institute of Mathematics NASU, Kyiv, Ukraine}
\maketitle
\fi

\begin{abstract}
We prove a theorem of Hinich type on existence of a model structure on a category related by an adjunction to the category of differential graded modules over a graded commutative ring.

    \ifx\chooseClass1
    \bigskip
    
Доведена теорема типу теореми Хініча про існування модельної структури на категорії пов’язаній спряженістю з категорією диференціально-градуйованих модулів над градуйованим комутативним кільцем.
    \fi
\end{abstract}

\ifx\chooseClass1
\maketitle

Institute of Mathematics NASU, Kyiv, Ukraine

\bigskip

Володимир Васильович Любашенко

\bigskip

Інститут математики\\
Національної академії наук України\\
вул. Терещенківська, 3\\
Київ-4, 01601 МСП

\bigskip

мешкає за адресою:\\
вул. Майорова 3, кв. 86,\\
Київ, 04201, Україна

\bigskip

Телефон: 430 9693~(дом), 235 7819~(сл)

\bigskip

lub@imath.kiev.ua

\newpage
\fi

\allowdisplaybreaks[1]

\section{Introduction}
Hinich proved in \cite{Hinich:q-alg/9702015} a theorem on existence of a model structure on a category related by an adjunction to the category of complexes.
In this article we give a detailed proof of a theorem of similar kind.
The two theorems differ at least in two points.
First, Hinich works with $\dg$\n-modules over a (commutative) ring, and we consider differential graded modules over a \emph{graded} commutative ring $\kk$.
Second, in the proof Hinich introduces certain morphisms which he calls elementary trivial cofibrations and shows that any trivial cofibration is a retract of countable composition of elementary ones.
We show that a trivial cofibration is a retract of an elementary trivial cofibration in our sense.

We apply our theorem to proving that categories of bi- or poly-modules over non-symmetric operads have a model structure \cite{Lyu-Ainf-Operad,Lyu-A8-several-entries}.
For modules over operads a model structure was constructed by Harper \cite[Theorem~1.7]{MR2593672}.
Since Hinich's article \cite{Hinich:q-alg/9702015} a plenty of results appeared in which given a (monoidal) model category one produces a model structure on another category related to the first category by an adjunction \cite[Section~2.5]{math.AT/0206094}, on category of monoids \cite[Theorem~3.1]{math.AT/9801082} or on the category of operads \cite[Remark~2]{math/0101102}, \cite[Theorem~1.1]{MR2821434}.
Clearly, in this approach one must have a model category to begin with.
The category of differential (unbounded) graded $\kk^0$\n-modules has a projective model structure for a commutative ring $\kk^0$ \cite{math.KT/0011216}.
The same result for \emph{graded commutative} ring $\kk$ has to be deduced from the case of commutative ring $\kk^0$ along the lines of \cite{math.AT/0206094}.
After that one has to prove that $\dg\text-\kk$-mod is a monoidal model category, which requires detailed information on cofibrations.
Such information is provided e.g. by the proof of Hinich type theorem: any cofibration is a retract of a countable composition of elementary cofibrations (of a concrete form).
Thus, a technical work does not seem to be avoidable in any approach.
One more reason to follow the Hinich's approach is pedagogical: it can be explained to students in detail as well as in examples.

\subsection{Notations and conventions.}
In this article `graded' means $\ZZ$\n-graded.
Let $\kk$ be a graded commutative ring (equipped with zero differential).
By \(\gr=\gr_\kk=\gr\text-\kk\modul\) we denote the closed category of $\ZZ$\n-graded $\kk$\n-modules with $\kk$\n-linear homomorphisms of degree 0.
Thus an object of $\gr$ is \(X=(X^m)^{m\in\ZZ}\).
Symmetry in the monoidal category of graded $\kk$\n-modules is chosen as \(c(x\tens y)=(-1)^{ml}y\tens x\) for \(x\in X^m\), \(y\in Y^l\).

The abelian category \(\dg=\dg\text-\kk\modul\) is the closed category of differential $\ZZ$\n-graded $\kk$\n-modules with chain $\kk$\n-linear homomorphisms.
Monomorphisms and epimorphisms of $\dg$ are componentwise injections and surjections.
A quasi-isomorphism \(M\to N\in\dg\) is a chain $\kk$\n-linear homomorphism inducing an isomorphism in homology.
For $a\in\ZZ$ the shift functor is defined by \([a]:\dg\to\dg\), \(M\mapsto M[a]\), \(M[a]^z=M^{z+a}\).
The shift functor extends componentwise to $\dg^S$ for any set $S$.

Denote by \(\sigma^a:M\to M[a]\) the ``identity map'' of degree \(\deg\sigma^a=-a\).
Write elements of $M[a]$ as \(m\sigma^a\).
When \(f:V\to X\) is a homogeneous map of certain degree, the map \(f[a]:V[a]\to X[a]\) is defined as \(f[a]=(-1)^{fa}\sigma^{-a}f\sigma^a\).
In particular, the differential \(d:M\to M\) of degree~1 in a $\dg$\n-module $M$ induces the differential \(d[a]=(-1)^a\sigma^{-a}d\sigma^a:M[a]\to M[a]\) in $M[a]$.
The degree 0 isomorphisms \(\sigma^{-a}\cdot(\sigma^a\tens1):(V\tens W)[a]\to(V[a])\tens W\),
\((v\tens w)\sigma^a\mapsto(-1)^{wa}v\sigma^a\tens w\), and \(\sigma^{-a}\cdot(1\tens\sigma^a):(V\tens W)[a]\to V\tens(W[a])\), \((v\tens w)\sigma^a\mapsto v\tens w\sigma^a\), are graded natural.
This means that for arbitrary homogeneous maps \(f:V\to X\), \(g:W\to Y\) the following squares commute:
\begin{diagram}[w=5em]
(V[a])\tens W &\lTTo^{\sigma^{-a}\cdot(\sigma^a\tens1)}_\sim &(V\tens W)[a]
&\rTTo^{\sigma^{-a}\cdot(1\tens\sigma^a)}_\sim &V\tens(W[a])
\\
\dTTo<{(f[a])\tens g} &&\dTTo<{(f\tens g)[a]} &&\dTTo>{f\tens(g[a])}
\\
(X[a])\tens Y &\lTTo^{\sigma^{-a}\cdot(\sigma^a\tens1)}_\sim &(X\tens Y)[a]
&\rTTo^{\sigma^{-a}\cdot(1\tens\sigma^a)}_\sim &X\tens(Y[a])
\end{diagram}
Actually, the second isomorphism is ``more natural'' than the first one, not only because it does not have a sign, but also because it suits better the right operator system of notations, accepted in this paper.
In the following we always identify \((V\tens W)[a]\) with \(V\tens(W[a])\) via \(\sigma^{-a}\cdot(1\tens\sigma^a)\).

Assume that \(\alpha:M\to N\in\dg\).
Denote by \(\Cone\alpha=(M[1]\oplus N,d_{\Cone})\in\Ob\dg\) the graded $\kk$\n-module with the differential
\[ d_{\Cone} =
\begin{pmatrix}
d_M[1] & \sigma^{-1}\alpha \\
0 & d_N
\end{pmatrix}
=
\begin{pmatrix}
-\sigma^{-1}d_M\sigma & \sigma^{-1}\alpha \\
0 & d_N
\end{pmatrix}.
\]

The following result generalizes a theorem of Hinich \cite[Section~2.2]{Hinich:q-alg/9702015}.

\begin{theorem}\label{thm-Hinich-model-category}
Suppose that $S$ is a set, a category $\cc$ is complete and cocomplete and \(F:\dg^S\rightleftarrows\cc:U\) is an adjunction.
Assume that $U$ preserves filtering colimits.
For any $x\in S$ consider the object $\KK_x$ of $\dg^S$, \(\KK_x(x)=\Cone(\id_\kk)\), \(\KK_x(y)=0\) for $y\ne x$. 
Assume that the chain map \(U(\inj_2):UA\to U(F(\KK_x[p])\sqcup A)\) is a quasi-isomorphism for all objects $A$ of $\cc$ and all $x\in S$, $p\in\ZZ$.
Equip $\cc$ with the classes of weak equivalences (resp. fibrations) consisting of morphisms $f$ of $\cc$ such that $Uf$ is a quasi-isomorphism (resp. an epimorphism).
Then the category $\cc$ is a model category.
\end{theorem}

\section{Proof of existence of model structure}
This section is devoted to proof of \thmref{thm-Hinich-model-category}, whose hypotheses we now assume.
The proof follows that of Hinich's theorem \cite[Section~2.2]{Hinich:q-alg/9702015} ideologically but not in details.
Constructions used in the proof describe cofibrations and trivial cofibrations in $\cc$. 

Denote the functor $U$ also by $-^\#$, \(UX=X^\#\) for $X\in\Ob\cc$ or $X\in\Mor\cc$.
Let \(\eps:FUA\to A\) be the adjunction counit and let \(\eta:M\to UFM\) be the adjunction unit.
The adjunction bijection is given by mutually inverse maps
\begin{align*}
(l:FM\to A) &\rMapsTo l^t =\bigl( M \rto\eta (FM)^\# \rto{l^\#} A^\# \bigr),
\\
\sS{^t}x =\bigl( FM \rto{Fx} F(A^\#) \rto\eps A \bigr) &\lMapsTo (x:M\to A^\#).
\end{align*}

Define three classes of morphisms in $\cc$:
\begin{align*}
\cW &=\{ f\in\Mor\cc \mid \forall\,x\in S\ f^\#(x) \textup{ is a quasi-isomorphism} \},
\\
\cR_f &=\{ f\in\Mor\cc \mid \forall\,x\in S\ \forall\,z\in\ZZ\ f^\#(x)^z \textup{ is surjective} \},
\\
\cL_c &= \sS{^\perp}\cR_{tf} \textup{ consists of maps $f\in\Mor\cc$ with the left lifting property}
\\
&\qquad \textup{with respect to all morphisms from } \cR_{tf}=\cW\cap\cR_f.
\end{align*}
We are going to prove that they are weak equivalences, fibrations and cofibrations of a certain model structure on $\cc$.

Let \(M\in\Ob\dg^S\), \(A\in\Ob\cc\), \(\alpha:M\to A^\#\in\dg^S\).
Denote by \(C=\Cone\alpha=(M[1]\oplus UA,d_{\Cone})\in\Ob\dg^S\) the cone taken pointwise, that is, for any $x\in S$ the complex \(C(x)=\Cone\bigl(\alpha(x):M(x)\to(UA)(x)\bigr)\) is the usual cone.
Denote by \(\bar\imath=\inj_2:UA\to C\) the obvious embedding.
Following Hinich \cite[Section~2.2.2]{Hinich:q-alg/9702015} define an object \(A\langle M,\alpha\rangle\in\Ob\cc\) as the pushout
\begin{diagram}[w=4em]
FU(A) & \rTTo^\eps &A
\\
\dTTo<{F\bar\imath} &&\dTTo>{\bar\jmath}
\\
FC &\rTTo^g &\NWpbk A\langle M,\alpha\rangle
\end{diagram}
Introduce a functor \(h_{A,\alpha}:\cc\to\Set\):
\[ h_{A,\alpha}(B) =\bigl\{ (f,t)\in \cc(A,B) \times \udg^S(M,B^\#)^{-1} \mid (t)d\equiv td_{B^\#} +d_Mt =\bigl( M \rto\alpha A^\# \rto{f^\#} B^\# \bigr) \bigr\}.
\]

\begin{lemma}\label{lem-AM-alpha-represent-hA-alpha}
The object \(D=A\langle M,\alpha\rangle\) and the element \((\bar\jmath,\theta)\in h_{A,\alpha}(D)\) represent the functor \(h_{A,\alpha}\), where
\[ \theta =\bigl( M \rto\sigma M[1] \rTTo^{\inj_1} C \rto\eta UFC \rto{Ug} UD \bigr).
\]
That is, the natural in $B$ transformation \(\psi_B:\cc(D,B)\to h_{A,\alpha}(B)\), \(1_D\mapsto(\bar\jmath,\theta)\), is bijective.
\end{lemma}

\begin{proof}
The boundary of degree $-1$ map \(h=\bigl(M\rto\sigma M[1]\rTTo^{\inj_1} C\bigr)\) is \((h)d=hd_C+d_Mh=\alpha\cdot\bar\imath\).
Therefore, \((\theta)d\) is the composition along the bottom path in the diagram
\begin{diagram}
M &\rTTo^\alpha &UA &\rTTo^\eta &UFUA &\rTTo^{U\eps} &UA
\\
&&\dTTo<{\bar\imath} &= &\dTTo<{UF\bar\imath} &= &\dTTo>{U\bar\jmath}
\\
&&C &\rTTo^\eta &UFC &\rTTo^{Ug} &UD
\end{diagram}
which equals to the top path, that is, to \(\alpha\cdot U\bar\jmath\).
Therefore, \((\bar\jmath,\theta)\in h_{A,\alpha}(D)\).
By the Yoneda lemma the natural transformation $\psi_B$ takes a morphism \(k:D\to B\) of $\cc$ to
\begin{equation}
h_{A,\alpha}(k)(\bar\jmath,\theta) =\bigl( A \rto{\bar\jmath} D \rto k B, M \rto h C \rto\eta (FC)^\# \rto{g^\#} D^\# \rto{k^\#} B^\# \bigr).
\label{eq-hA-alpha(k)(j-theta)}
\end{equation}

Let us prove injectivity of $\psi_B$.
Let \(k_1,k_2:D\to B\) satisfy
\[ (f_1,t_1) \equiv h_{A,\alpha}(k_1)(\bar\jmath,\theta) = h_{A,\alpha}(k_2)(\bar\jmath,\theta) \equiv (f_2,t_2).
\]
Then
\[ \bigl( M[1] \rTTo^{\inj_1} C \rto\eta (FC)^\# \rto{g^\#} D^\# \rto{k_p^\#} B^\# \bigr) =\sigma^{-1}t_p
\]
does not depend on $p=1,2$.
On the other summand of $C$ we also have that
\[ \bigl( A^\# \rto{\bar\imath} C \rto\eta (FC)^\# \rto{g^\#} D^\# \rto{k_p^\#} B^\# \bigr) =\bigl( A^\# \rto{\bar\jmath^\#} D^\# \rto{k_p^\#} B^\# \bigr) =f_p^\#
\]
does not depend on $p=1,2$.
Therefore,
\[ l_p^t =\bigl( C \rto\eta (FC)^\# \rto{g^\#} D^\# \rto{k_p^\#} B^\# \bigr)
\]
also does not depend on $p=1,2$. Their adjuncts \(l_p=\bigl(FC\rto g D\rto{k_p} B \bigr)\) must not depend on $p$ either.
By assumption
\[ \bigl( A \rto{\bar\jmath} D \rto{k_1} B \bigr) =f_1 =f_2 =\bigl( A \rto{\bar\jmath} D \rto{k_2} B \bigr).
\]
The pushout property of $D$ allows only one morphism $D\to B$ with such properties, hence, $k_1=k_2$.

Let us prove surjectivity of $\psi_B$.
Given an element \((f:A\to B,t:M\to B^\#)\in h_{A,\alpha}(B)\) we construct a degree $0$ map \(x:C\to B^\#\)
\[ x= \biggl(\begin{matrix}M[1] \rto{\sigma^{-1}} M \rto t B^\# \\ A^\# \rto{f^\#} B^\# \end{matrix}\biggr).
\]
One easily checks that $x$ is a chain map, $x\in\dg^S$.
Its adjunct is denoted
\[ l =\sS{^t}x =\bigl(FC \rto{Fx} F(B^\#) \rto\eps B \bigr).
\]
Since \(\bar\imath\cdot x=f^\#:A^\#\to B^\#\), we have
\[ F\bar\imath\cdot l =\bigl(F(A^\#) \rTTo^{F(f^\#)} F(B^\#) \rto\eps B \bigr) =\eps\cdot f.
\]
By definition of pushout $D$ there exists a unique morphism \(k:D\to B\in\cc\) such that \(f=\bar\jmath\cdot k\), \(l=g\cdot k\).
Hence,
\begin{gather*}
x =l^t =\bigl( C \rto\eta (FC)^\# \rto{l^\#} B^\# \bigr),
\\
t =\bigl( M \rto\sigma M[1] \rTTo^{\inj_1} C \rto x B^\# \bigr) =\bigl( M \rto\sigma M[1] \rTTo^{\inj_1} C \rto\eta (FC)^\# \rto{g^\#} D^\# \rto{k^\#} B^\# \bigr).
\end{gather*}
Therefore, \(\psi_B(k)=(f,t)\) and $\psi_B$ is bijective.
\end{proof}

\begin{corollary}
The map \(\bigl(M\rto\alpha A^\#\rto{\bar\jmath^\#} A\langle M,\alpha\rangle^\#\bigr)=(\theta)d\) is null-homotopic.
If $d_M=0$, then for any cycle $m\in ZM$ the cycle \(m\alpha\in ZA^\#\) is taken by \(\bar\jmath^\#\) to the boundary of the element \(m\theta\in A\langle M,\alpha\rangle^\#\).
\end{corollary}

Thus, when \(F:\dg^S\to\cc\) is the functor of constructing a free $\dg$\n-algebra of some kind, the maps $\bar\jmath$ are interpreted as ``adding variables to kill cycles''.

The following statement is well-known.

\begin{lemma}\label{lem-surjective-quasi-isomorphism}
Assume that \(g:U\to V\in\Com\) is a surjective quasi-isomorphism.
Then for any pair \((u,v)\), \(u\in U^{n+1}\), \(v\in V^n\), such that \(ud=0\), \(ug=vd\) there is an element \(w\in U^n\) such that \(wd=u\), \(wg=v\).
\end{lemma}

\begin{proof}
Vanishing of \(H^{n+1}(g)[u]=[gu]=0\) implies vanishing of the cohomology class \([u]=0\).
There is \(y\in U^n\) such that \(yd=u\).
Denote \(c=yg\in V^n\), then
\[ cd =ygd =ydg =ug =vd.
\]
Hence, $c-v$ is a cycle, and there is a cycle \(z\in Z^nU\) such that \([zg]=[c-v]\).
There is \(e\in V^{n-1}\) such that \(zg=c-v+ed\).
The element $e$ lifts to \(x\in U^{n-1}\) such that \(xg=e\).
Thus,
\[ yg =c =zg -xgd +v =(z -xd)g +v.
\]
Therefore, \(w=y-z+xd\) satisfies \(wg=v\) and \(wd=u\).
\end{proof}

We say that $M$ consists of free $\kk$\n-modules if for any $x\in S$ the graded $\kk$\n-module $M(x)$ is free -- isomorphic to \(\oplus_{a\in\ZZ}P^a\kk[a]\) for some graded set $P$ and $d_M=0$.

\begin{proposition}
Let $M$ consist of free $\kk$\n-modules, $d_M=0$, $A\in\Ob\cc$ and \(\alpha:M\to A^\#\in\dg^S\).
Then \(\bar\jmath:A\to A\langle M,\alpha\rangle\in\cL_c\).
\end{proposition}

\begin{proof}
Let the image $y^\#$ of a morphism \(y:U\to V\in\cc\) be an epimorphism and a quasi-isomorphism.
Let \(u:A\to U\in\cc\).
Morphisms \(v:A\langle M,\alpha\rangle\to V\), which make the square
\begin{diagram}[LaTeXeqno]
A &\rTTo^u &U
\\
\dTTo<{\bar\jmath} &\ruDashTo^w &\dFib~\wr>y
\\
A\langle M,\alpha\rangle &\rTTo^v &V
\label{dia-w-jv-uy}
\end{diagram}
commutative, are in bijection with elements \((A\rto u U\rto y V,M\rto t V^\#)\in h_{A,\alpha}(V)\).
Thus,
\[ (t)d =d_Mt+td_{V^\#} =\bigl(M\rto\alpha A^\#\rto{u^\#} U^\#\rto{y^\#} V^\#\bigr).
\]
For some graded set \(P=(P^a(s)\mid a\in\ZZ,\,s\in S)\), \(P^a(s)\in\Set\), we have \(M=P\kk=\bigl(\oplus_{a\in\ZZ}P^a(s)\kk[a]\bigr)_{s\in S}\).
Let us denote the chosen basis of $M$ by \((e_p)_{p\in P^\bull(\bull)}\), \(\deg e_p=\deg p\).
For an arbitrary \(p\in P^a(s)\) denote $n=a-1$.
We have a cycle \(e_p\alpha u^\#\in Z^{n+1}(U^\#)\) and an element \(e_pt\in(V^\#)^n\) such that \((e_p\alpha u^\#)y^\#=(e_pt)d_{V^\#}\).
By \lemref{lem-surjective-quasi-isomorphism} there is an element denoted \((e_pr)\in(U^\#)^n\) such that \(e_p\alpha u^\#=(e_pr)d_{U^\#}\) and \(e_pt=(e_pr)y^\#\).
Choosing such \(e_pr\) for all \(p\in P^\bull(\bull)\) we get a map \(r\in\udg^S(M,U^\#)^{-1}\) such that the triangles commute
\[
\begin{diagram}[inline]
A^\# &\rTTo^{u^\#} &U^\#
\\
\uTTo<\alpha &\ruTTo>{(r)d} &
\\
M &&
\end{diagram}
\quad,\quad
\begin{diagram}[inline]
&&U^\#
\\
&\ruTTo<r &\dTTo>{y^\#}
\\
M &\rTTo^t &V^\# 
\end{diagram}
\quad.
\]
Thus a pair \((u:A\to U,r:M\to U^\#)\in h_{A,\alpha}(U)\) determines a morphism \(w:A\langle M,\alpha\rangle\to U\in\cc\) by \lemref{lem-AM-alpha-represent-hA-alpha}.
Due to \eqref{eq-hA-alpha(k)(j-theta)} the equation
\[ u =\bigl( A \rto{\bar\jmath} A\langle M,\alpha\rangle \rto w U \bigr)
\]
holds.
Naturality of bijection $\psi$,
\begin{diagram}
h_{A,\alpha}(U) &\rTTo^{\psi_U}_\sim &\cc(A\langle M,\alpha\rangle,U)
\\
\dTTo<{h_{A,\alpha}(U)} &= &\dTTo>{\cc(1,y)}
\\
h_{A,\alpha}(V) &\rTTo^{\psi_V}_\sim &\cc(A\langle M,\alpha\rangle,V)
\end{diagram}
applied to the pair $(u,r)$ gives
\begin{diagram}
(u:A\to U,r:M\to U^\#) &\rMapsTo &w
\\
\dMapsTo<{(-\cdot y,-\cdot y^\#)} &= &\dMapsTo>{-\cdot y}
\\
(uy:A\to V,ry^\#:M\to V^\#) =(\bar\jmath v,t) &\rMapsTo &v =wy.
\end{diagram}
This gives another equation
\[ v =\bigl( A\langle M,\alpha\rangle \rto w U \rto y V \bigr)
\]
and $w$ is the sought diagonal filler for \eqref{dia-w-jv-uy}.
\end{proof}

If $M$ consists of free $\kk$\n-modules (and $d_M=0$), then \(\bar\jmath:A\to A\langle M,\alpha\rangle\) is a cofibration.
It might be called an \emph{elementary standard cofibration}.
If
\[ A \to A_1 \to A_2 \to \cdots
\]
is a sequence of elementary standard cofibrations, $B$ is a colimit of this diagram, then the ``infinite composition'' map $A\to B$ is a cofibration called a \emph{standard cofibration} \cite[Section~2.2.3]{Hinich:q-alg/9702015}.

\begin{lemma}\label{lem-alpha-alpha}
Let \(\alpha\sim\alpha':M\to A^\#\).
Then there is a natural in $B$ bijection \(h_{A,\alpha}(B)\simeq h_{A,\alpha'}(B)\).
Hence, there is an isomorphism $k$ of representing objects, which is the last arrow in the equation which holds in $\cc$:
\[ \bar\jmath' =\bigl( A \rto{\bar\jmath} A\langle M,\alpha\rangle \rto[\sim]{k} A\langle M,\alpha'\rangle \bigr).
\]
\end{lemma}

\begin{proof}
Let \(h\in\udg^S(M,A^\#)^{-1}\) be a homotopy, \(\alpha-\alpha'=hd+dh:M\to A^\#\).
Then we have well defined maps
\begin{gather*}
h_{A,\alpha}(B) =\bigl\{ (f:A\to B,t:M\to B^\#) \mid (t)d=\alpha f^\# \bigr\}
\\
\begin{diagram}[w=4em]
(f,t) &&(f,q+hf^\#)
\\
\dMapsTo &&\uMapsTo
\\
(f,t-hf^\#) &&(f,q)
\end{diagram}
\\
h_{A,\alpha'}(B) =\bigl\{ (f:A\to B,q:M\to B^\#) \mid (q)d=\alpha' f^\# \bigr\}
\end{gather*}
since
\begin{align*}
(t-hf^\#)d =\alpha f^\# -(\alpha-\alpha')f^\# =\alpha'f^\#,
\\
(q+hf^\#)d =\alpha'f^\# +(\alpha-\alpha')f^\# =\alpha f^\#.
\end{align*}
These maps are mutually inverse and natural in $B$.

Take \(B=A\langle M,\alpha'\rangle\).
There is a commutative square of bijections
\begin{diagram}[w=6.4em,tight]
\cc(A\langle M,\alpha'\rangle,A\langle M,\alpha'\rangle) &\rTTo^{\cc(k,1)}_\sim &\cc(A\langle M,\alpha\rangle,A\langle M,\alpha'\rangle)
\\
\dTTo<\psi>\wr &&\dTTo<\wr>\psi
\\
h_{A,\alpha'}(A\langle M,\alpha'\rangle) &\rTTo^\sim &h_{A,\alpha}(A\langle M,\alpha'\rangle)
\end{diagram}
which gives the equation
\begin{diagram}[w=3.2em,tight]
1_B &&\rMapsTo &&k
\\
\dMapsTo &&&&\dMapsTo
\\
(\bar\jmath',t') &\rMapsTo &(\bar\jmath',t'+h\bar\jmath'^\#) &\rEq &(\bar\jmath k,tk^\#).
\end{diagram}
In particular, \(\bar\jmath'=\bar\jmath k\).
\end{proof}

\begin{remark}\label{rem-pushout-pushout-A<M-alpha>}
Consider a diagram \(\alpha'=\bigl( M' \rto\beta M'' \rto{\alpha''} A^\# \bigr)\) in $\dg^S$.
These morphisms lead to natural transformation \(h_{A,\alpha''}(B)\to h_{A,\alpha'}(B)\), \((f,t)\mapsto(f,\beta\cdot t)\), or equivalently \(\cc(A\langle M'',\alpha''\rangle,B)\to\cc(A\langle M',\alpha'\rangle,B)\), which comes from a unique morphism \(A\langle\beta\rangle:A\langle M',\beta\cdot\alpha''\rangle\to A\langle M'',\alpha''\rangle\in\cc\).
It can be found from the diagram
\begin{diagram}[w=4em,LaTeXeqno]
&&F(A^\#) &\rTTo^\eps &A
\\
&\ldTTo(2,4)<{F\bar\imath''} &\dTTo>{F\bar\imath'} &&\dTTo<{\bar\jmath'} &\rdTTo(2,4)>{\bar\jmath''}
\\
&&F(C') &\rTTo^{g'} &\NWpbk A\langle M',\alpha'\rangle
\\
&\ldTTo>{F\gamma} &&&&\rdTTo_{A\langle\beta\rangle}
\\
F(C'') &&&\rTTo^{g''} &&&A\langle M'',\alpha''\rangle
\label{dia-pushout-pushout-A<M-alpha>}
\end{diagram}
where \(\gamma=\Cone(\beta,1):C'\to C''\) is the morphism of cones, induced by $\beta$.

In fact, put \(B=A\langle M'',\alpha''\rangle\).
The unit morphism $1_B$ corresponds to \((\bar\jmath'',\theta'')\in h_{A,\alpha''}(B)\) which is taken to \((\bar\jmath'',\beta\cdot\theta'')\in h_{A,\alpha'}(B)\).
The latter element has to coincide with \((\bar\jmath'\cdot A\langle\beta\rangle,\theta'\cdot A\langle\beta\rangle^\#)\).
The equation \(\bar\jmath''=\bar\jmath'\cdot A\langle\beta\rangle\) is the right triangle of \eqref{dia-pushout-pushout-A<M-alpha>}.
The equation \(\beta\cdot\theta''=\theta'\cdot A\langle\beta\rangle^\#\) can be written as the exterior of
\begin{diagram}
M' &\rTTo^\sigma &M'[1] &\rTTo^{\inj_1} &C' &\rTTo^{g'^t} &D'^\#
\\
\dTTo<\beta &&&&\dTTo<\gamma &&\dTTo>{A\langle\beta\rangle^\#}
\\
M'' &\rTTo^\sigma &M''[1] &\rTTo^{\inj_1} &C'' &\rTTo^{g''^t} &D''^\#
\end{diagram}
The mentioned right triangle implies commutativity of the exterior of
\begin{diagram}
A^\# &\rTTo^{\bar\imath'} &C' &\rTTo^{g'^t} &D'^\#
\\
\dEq &&\dTTo<\gamma &&\dTTo>{A\langle\beta\rangle^\#}
\\
A^\# &\rTTo^{\bar\imath''} &C'' &\rTTo^{g''^t} &D''^\#
\end{diagram}
This fact jointly with the previous implies commutativity of the right square, which is equivalent to lower trapezia in \eqref{dia-pushout-pushout-A<M-alpha>}.

In particular, for \(0=\bigl(0\rto0 M\rto\alpha A^\#\bigr)\) we have \(\bar\imath'=\id:A^\#\to C'\), \(\bar\jmath'=\id:A\to A\langle0,0\rangle\), \(\bar\jmath''=\bar\jmath=A\langle0\rangle:A=A\langle0,0\rangle\to A\langle M,\alpha\rangle\).
\end{remark}

\begin{remark}\label{rem-A<M0>}
For \(0:M\to A^\#\) we have that \(A\langle M,0\rangle\simeq F(M[1])\sqcup A\) and $\bar\jmath=\inj_2$ is the canonical embedding.
In fact, \(C=M[1]\oplus A^\#\) is the direct sum of complexes and \(A\langle M,0\rangle\) is found from the following diagram
\begin{diagram}
F(A^\#) &&\rTTo^\eps &&A
\\
\dTTo<{F(\inj_2)} &\rdTTo>{\inj_2} &&&\dTTo>{\inj_2}
\\
F(M[1]\sqcup A^\#) &\rTTo^\sim &F(M[1])\sqcup F(A^\#) &\rTTo^{1\sqcup\eps} &F(M[1])\sqcup\NWpbk A=A\langle M,0\rangle
\end{diagram}
\end{remark}

\begin{example}\label{exa-eta:N-FN}
Let \(N\in\Ob\dg^S\).
Take $FN$ for $A$ and \(\eta:N\to(FN)^\#\) for $\alpha$.
We claim that we can take \(F(\Cone1_N)\) for \((FN)\langle N,\eta\rangle\).
In fact,
\begin{align*}
h_{FN,\eta}(B) &=\bigl\{ (f:FN\to B,t:N\to B^\#) \mid (t)d =\eta\cdot f^\# \bigr\} =\bigl\{ (f,t) \mid (t)d =f^t \bigr\}
\\
&=\bigl\{ (f,t) \mid f=\sS{^t}((t)d) \bigr\} =\bigl\{ t\in\udg^S(N,B^\#)^{-1} \bigr\} \simeq \udg^S(N[1],B^\#)^0
\\
&\overset{(!)}\simeq \dg^S\bigl((N[1]\oplus N,d_{\Cone1_N}),B^\#\bigr) =\dg^S(\Cone1_N,B^\#) \simeq \cc(F(\Cone1_N),B).
\end{align*}
Bijection (!) is left to the reader as an exercise.
\end{example}

\begin{proposition}\label{pro-N-free-j-composition-2-elementary-standard-cofibrations}
Let \(N=P\kk\in\dg^S\) consist of free $\kk$\n-modules, $d_N=0$, and \(M=\Cone1_{N[-1]}=(N\oplus N[-1],d_{\Cone})\).
Then for any morphism \(\alpha:M\to UA\in\dg^S\) the morphism \(\bar\jmath:A\to A\langle M,\alpha\rangle\) is a standard cofibration, composition of two elementary standard cofibrations.
\end{proposition}

\begin{proof}
The complex $M$ is contractible, hence, \(\alpha\sim0=\alpha':M\to A^\#\).
Applying \lemref{lem-alpha-alpha} we find that \(\bar\jmath=\bigl(A\rTTo^{A\langle0\rangle} A\langle M,0\rangle\rto\sim A\langle M,\alpha\rangle \bigr)\), thus it suffices to prove the claim for $\alpha=0$.

The embedding \(\inj_2:N[-1]\to M\) induces a diagram
\begin{diagram}
A &\rTTo^{\bar\jmath'} &A\langle N[-1],0\rangle &\rEq &A\langle N[-1],0\rangle\langle0,0\rangle
\\
&\rdTTo_{\bar\jmath''} &\dTTo>{A\langle\inj_2\rangle} &&\dTTo>{A\langle N[-1],0\rangle\langle0\rangle}
\\
&&A\langle M,0\rangle &\rTTo^\sim &A\langle N[-1],0\rangle\langle N,\eta\rangle
\end{diagram}
Commutativity of the triangle is contained in diagram~\eqref{dia-pushout-pushout-A<M-alpha>}.
Commutativity of the square is implied by \remref{rem-A<M0>}, which gives \(A\langle N[-1],0\rangle=FN\sqcup A\), and by the equation
\begin{diagram}
FN &\rEq &(FN)\langle0,0\rangle
\\
\dTTo<{F(\inj_2)} &= &\dTTo>{(FN)\langle0\rangle}
\\
F(\Cone1_N) &\rEq &(FN)\langle N,\eta\rangle.
\end{diagram}
The latter equation follows from \exaref{exa-eta:N-FN}.
Let us take \(B=F(\Cone1_N)\) in it and find the element of \(h_{FN,\eta}(F(\Cone1_N))\), which goes into $1_B$ under the sequence of bijections considered in the example.
Moving backwards we find elements
\begin{multline*}
1_B \mapsto \langle \eta: \Cone1_N \to (F\Cone1_N)^\# \rangle \mapsto \bigl\langle N[1] \rTTo^{\inj_1} \Cone1_N \rto\eta (F\Cone1_N)^\# \bigr\rangle
\\
\mapsto t=\bigl\langle N \rto\sigma N[1] \rTTo^{\inj_1} \Cone1_N \rto\eta (F\Cone1_N)^\# \bigr\rangle \in \udg^S(N,B^\#)^{-1}.
\end{multline*}
Computation in the proof of \lemref{lem-AM-alpha-represent-hA-alpha} give
\[ (t).d =\bigl\langle N \rTTo^{\inj_2} \Cone1_N \rto\eta (F\Cone1_N)^\# \bigr\rangle =\bigl\langle N \rto\eta (FN)^\# \rTTo^{(F\inj_2)^\#} (F\Cone1_N)^\# \bigr\rangle,
\]
hence $t$ comes from the pair \((F(\inj_2),t)\in h_{FN,\eta}(F(\Cone1_N))\).
Thus, \(\bar\jmath'':A\to A\langle M,0\rangle\) is a composition of two elementary standard cofibrations and a standard cofibration itself.
\end{proof}

\begin{proposition}\label{pro-N-dg-submodule-cycles-Cone}
Let \(r:A\to Y\in\cc\).
Denote by
\begin{align*}
N &=Z\Cone(r^\#[-1]:A^\#[-1]\to Y^\#[-1])
\\
&=\bigl\{ (u,y\sigma^{-1})\in A^\#\times Y^\#[-1] \mid ud=0,\, ur^\# -yd_{Y^\#} =0 \bigr\}
\end{align*}
the differential graded $\kk$\n-submodule of cycles of \(\Cone(r^\#[-1])\), $d_N=0$.
Denote by \(\pr_1:N\to A^\#\in\dg^S\) (resp. \(\pr_2:N\to Y^\#[-1]\in\gr^S\)) the map \((u,y\sigma^{-1})\mapsto u\) (resp. \((u,y\sigma^{-1})\mapsto y\sigma^{-1}\)).
Define \(D=A\langle N,\pr_1\rangle\).
Then
\[ \bigl( r:A\to Y, t=(N \rTTo^{\pr_2} Y^\#[-1] \rto\sigma Y^\#) \bigr)
\]
is an element of \(h_{A,\pr_1}(Y)\).
The corresponding morphism \(q:D\to Y\) satisfies \(r=\bigl(A\rto{\bar\jmath} A\langle N,\pr_1\rangle \rto q Y\bigr)\).
The composition
\[ \beta =\bigl\langle N \hookrightarrow \Cone(r^\#[-1]) \rTTo^{\Cone(\bar\jmath^\#[-1],1)} \Cone(q^\#[-1]) \bigr\rangle,\quad \Cone(\bar\jmath^\#[-1],1) = \Bigl(\begin{matrix}\bar\jmath^\# &0 \\ 0 &1\end{matrix}\Bigr),
\]
is null-homotopic, \(\beta=(\theta,0).d=(\theta,0)\cdot d_{\Cone(q^\#[-1])}\), thus, all cycles of \(\Cone(r^\#[-1])\) are taken by \(\Cone(\bar\jmath^\#[-1],1_{Y^\#[-1]})\) to boundaries in \(\Cone(q^\#[-1])\).
\end{proposition}

\begin{proof}
Let us show that \((r,t)\in h_{A,\pr_1}(Y)\).
In fact, the diagram
\begin{diagram}[w=3.5em,tight]
N &\rTTo^{\pr_2} &Y^\#[-1] &\rTTo^\sigma &Y^\#
\\
\dTTo<{\pr_1} &&&&\dTTo>{d_{Y^\#}}
\\
A^\# &&\rTTo^{r^\#} &&Y^\#
\end{diagram}
commutes as the computation shows
\begin{diagram}[w=3.5em,tight]
(u,y\sigma^{-1}) &\rMapsTo &y\sigma^{-1} &\rMapsTo &y
\\
\dMapsTo &&&&\dMapsTo
\\
u &&\rMapsTo &&ur^\#=yd_{Y^\#}
\end{diagram}
The corresponding morphism \(q:D\to Y\) satisfies \((r,t)=\bigl(\bar\jmath\cdot q,N\rto\theta D^\#\rto{q^\#} Y^\#\bigr)\) by \eqref{eq-hA-alpha(k)(j-theta)}.

One can easily check that cones are related by the chain map
\[ \Cone(\bar\jmath^\#[-1],1_{Y^\#[-1]}) =\Bigl(\begin{matrix}\bar\jmath^\# &0 \\ 0 &1_{Y^\#[-1]}\end{matrix}\Bigr): \Cone((\bar\jmath^\#q^\#)[-1]) \to \Cone(q^\#[-1]).
\]
The composition $\beta$ takes \((u,y\sigma^{-1})\in N\) to \((u\bar\jmath^\#,y\sigma^{-1})\in\Cone(q^\#[-1])\).
Since $d_N=0$ the map
\[ (\theta,0).d =(\theta,0)\Bigl(\begin{matrix}d_{D^\#} &q^\#\sigma^{-1} \\ 0 &d_{Y^\#[-1]}\end{matrix}\Bigr) =(\pr_1\cdot\bar\jmath^\#,\theta q^\#\sigma^{-1}) =(\pr_1\cdot\bar\jmath^\#,t\sigma^{-1}) =(\pr_1\cdot\bar\jmath^\#,\pr_2)
\]
takes \((u,y\sigma^{-1})\) to the same \((u\bar\jmath^\#,y\sigma^{-1})\) as $\beta$.
\end{proof}

Assume hypotheses of \thmref{thm-Hinich-model-category} hold.

\begin{proposition}\label{pro-N-free-j-weak-equivalence}
Let \(N=P\kk\in\dg^S\) consist of free $\kk$\n-modules, $d_N=0$, and \(M=\Cone1_{N[-1]}\).
Then for all \(\alpha:M\to A^\#\in\dg^S\) the morphism \(\bar\jmath:A\to A\langle M,\alpha\rangle\) is in $\cw$.
\end{proposition}

\begin{proof}
The complex $M$ is contractible, hence, it suffices to consider $\alpha=0$.
Consider the directed set of finite graded subsets $Q\subset P$ (that is, the set \(\bigsqcup_{c\in\ZZ}^{x\in S}Q^c(x)\) is finite).
We have
\begin{align*}
M[1] &=P\KK[1] =\bigoplus_{c\in\ZZ}^{x\in S} P^c(x)\KK_x[c+1] =\colim_{Q\subset P} \coprod_{x\in S,\,c\in\ZZ}^{q\in Q^c(x)}\KK_x[c+1],
\\
\bar\jmath^\# &=\inj_2^\# =\bigl\langle A^\# \to (F(M[1])\ttt\coprod A)^\# \bigr\rangle
\\
&= \Bigl\langle A^\# \to \Bigl(\colim_{Q\subset P} \bigl(\coprod_{x\in S,\,c\in\ZZ}^{q\in Q^c(x)} F(\KK_x[c+1])\bigr)\coprod A\Bigr)^\# \Bigr\rangle
\\
&= \Bigl\langle A^\# \to \colim_{Q\subset P} \Bigl(\bigl(\coprod_{x\in S,\,c\in\ZZ}^{q\in Q^c(x)} F(\KK_x[c+1])\bigr)\coprod A\Bigr)^\# \Bigr\rangle.
\end{align*}
For any finite $Q$ the map \(\inj_2^\#:A^\#\to\bigl(\bigl(\coprod_{x\in S,\,c\in\ZZ}^{q\in Q^c(x)} F(\KK_x[c+1])\bigr)\coprod A\bigr)^\#\) is a quasi-\hspace{0pt}isomorphism as a finite composition of quasi-isomorphisms.
Thus its cone is acyclic. Therefore, the cone
\begin{multline*}
\Cone \Bigl\langle \bar\jmath^\#: A^\# \to \colim_{Q\subset P} \Bigl(\bigl(\coprod_{x\in S,\,c\in\ZZ}^{q\in Q^c(x)} F(\KK_x[c+1])\bigr)\coprod A\Bigr)^\# \Bigr\rangle
\\
\simeq \colim_{Q\subset P} \Cone \Bigl\langle A^\# \to \Bigl(\bigl(\coprod_{x\in S,\,c\in\ZZ}^{q\in Q^c(x)} F(\KK_x[c+1])\bigr)\coprod A\Bigr)^\# \Bigr\rangle
\end{multline*}
is acyclic and \(\bar\jmath^\#\) is a quasi-isomorphism.
\end{proof}

To sum up Propositions \ref{pro-N-free-j-composition-2-elementary-standard-cofibrations} and \ref{pro-N-free-j-weak-equivalence} assume that \(N\in\Ob\dg^S\) consists of free $\kk$\n-modules, $d_N=0$, and \(M=\Cone1_{N[-1]}=(N\oplus N[-1],d_{\Cone})\).
Then for any morphism \(\alpha:M\to UA\in\dg^S\) the morphism \(\bar\jmath:A\to A\langle M,\alpha\rangle\) is a trivial cofibration in $\cc$ and a standard cofibration, composition of two elementary standard cofibrations.
It is called a \emph{standard trivial cofibration}.

\begin{proof}[Proof of \thmref{thm-Hinich-model-category}]
(MC1) (Co)completeness of $\cc$ is assumed.
Axioms (MC2) (three-for-two for $\cw$) and (MC3) (closedness of $\cL_c$, $\cw$, $\cR_f$ with respect to retracts) are obvious.
The class $\cL_c$ is \(\sS{^\perp}(\cW\cap\cR_f)\) by definition.

\subsubsection*{\textup{(MC5)(ii)} Functorial factorization into a trivial cofibration and a fibration.}
Let \(f:X\to Y\in\cc\).
Denote \(N=Y^\#\kk\), \(M[1]=\Cone1_{N[-1]}=(N\oplus N[-1],d_{\Cone})\simeq Y^\#\KK[-1]\).
The $\kk$\n-linear degree 0 map \(N\to Y^\#\), \(e_y\mapsto y\), extends in a unique way to a degreewise surjection \(\pi_Y^t:M[1]\to Y^\#\in\dg^S\), which determines a morphism \(\pi_Y:F(M[1])\to Y\in\cc\).
Combining it with the previous we get a morphism \(\pi_Y\cup f:F(M[1])\coprod X\to Y\in\cc\).
Since \(\pi_Y^t=\bigl\langle M[1]\rto\eta (F(M[1]))^\#\rto{\pi_Y^\#} Y^\#\bigr\rangle\) is a surjection, the map \(\pi_Y^\#=\bigl\langle(F(M[1]))^\#\rTTo^{\inj_1^\#} (F(M[1])\coprod X)^\#\rTTo^{(\pi_Y\cup f)^\#} Y^\#\bigr\rangle\) is a surjection as well.
Therefore, \((\pi_Y\cup f)^\#\) is a surjection and \(\pi_Y\cup f\in\cR_f\).
The decomposition
\[ f =\bigl( X \rto{\bar\jmath} X\langle M,0\rangle =F(M[1])\ttt\coprod X \rTTo^{(\pi_Y\cup f)^\#} Y\bigr)
\]
into a trivial cofibration and a fibration is functorial in $f$.

\subsubsection*{\textup{(MC5)(i)} Functorial factorization into a cofibration and a trivial fibration.}
Let us construct inductively the following diagram in $\cc$
\begin{diagram}[w=2em,LaTeXeqno]
X &\rEq &D_0 &\rTTo^{h_0} &D_1 &\rTTo^{h_1} &D_2 &\rTTo^{h_2} &\dots
\\
&&&\rdTTo(4,2)_{f=q_0} &&\rdTTo^{q_1} &\dTTo>{q_2} &\dots
\\
&&&&&&Y
\label{dia-XDDD-Y}
\end{diagram}
so that all $h_i$ were cofibrations.
Given $q_n$ for $n\ge0$ denote
\begin{align*}
N_n &=Z\Cone(q_n^\#[-1]:D_n^\#[-1]\to Y^\#[-1])
\\
&=\bigl\{ (u,y\sigma^{-1})\in D_n^\#\times Y^\#[-1] \mid ud=0,\, uq_n^\# -yd_{Y^\#} =0 \bigr\}
\end{align*}
as in \propref{pro-N-dg-submodule-cycles-Cone}.
Being a subset of cycles $N_n$ is a graded $\kk$\n-submodule with $d_{N_n}=0$.
Viewing $N_n$ as a graded set introduce a graded $\kk$\n-module \(M_n=N_n\kk\), $d_{M_n}=0$, with the projection \(p_n:M_n\rEpi N_n\in\dg^S\), \(e_v\mapsto v\) for all \(v\in N_n^\bull(\bulle)\).
Let us denote \(\alpha_n=\bigl(M_n\rto{p_n} N_n\rTTo^{\pr_1} D_n^\#\bigr)\in\dg^S\).
Choose \(D_{n+1}=D_n\langle M_n,\alpha_n\rangle\), then \(h_n=D_n\langle0\rangle:D_n\to D_{n+1}\) is a cofibration.
\propref{pro-N-dg-submodule-cycles-Cone} and \remref{rem-pushout-pushout-A<M-alpha>} imply that \((q_n:D_n\to Y,t_n=\bigl(M_n\rto{p_n} N_n\rTTo^{\pr_2} Y^\#[-1]\rto\sigma Y^\#\bigr)\) is an element of \(h_{D_n,\alpha_n}(Y)\).
A morphism \(q_{n+1}:D_{n+1}=D_n\langle M_n,\alpha_n\rangle\to Y\in\cc\) corresponds to the pair \((q_n,t_n)\) such that \(q_n=\bigl(D_n\rto{h_n} D_{n+1}\rTTo^{q_{n+1}} Y\bigr)\) in $\cc$, which gives the required diagram.

Let us prove that \(q_2^\#:D_2^\#\to Y^\#\) is surjective in all degrees.
Let \(y\in Y^{\#\bull}(\bulle)\).
Then \((0,yd\sigma^{-1})\in N_0\), \(e_{(0,yd\sigma^{-1})}\in M_0\), \(e_{(0,yd\sigma^{-1})}\theta_0\in D_1^\#\).
The equation \(\theta_0q_1^\#=t_0=p_0\cdot\pr_2\cdot\sigma:M_0\to Y^\#\) implies that
\[ e_{(0,yd\sigma^{-1})}\theta_0q_1^\# -yd_{Y^\#} =(0,yd\sigma^{-1})\pr_2\sigma -yd =0.
\]
Furthermore,
\[ e_{(0,yd\sigma^{-1})}\theta_0d_{D_1^\#} =e_{(0,yd\sigma^{-1})}.(\theta)d =e_{(0,yd\sigma^{-1})}\alpha_0\bar\imath_0\eta g_0^\# =(0,yd\sigma^{-1})\pr_1\alpha_0\bar\imath_0\eta g_0^\# =0.
\]
Thus, \((e_{(0,yd\sigma^{-1})}\theta_0,y\sigma^{-1})\in N_1\).
Therefore the map \(\pr_2\cdot\sigma:N_1\to Y^\#\) is surjective in each degree.
Hence, the map \(t_1=\bigl(M_1\rEpi^{p_1} N_1\rTTo^{\pr_2} Y^\#[-1]\rto\sigma Y^\#\bigr)\) is surjective as well.
Since \(t_1=\bigl(M_1\rto{\theta_1} D_2^\#\rto{q_2^\#} Y^\#\bigr)\), it follows that $q_2^\#$ is surjective in each degree.
Consequently \(q_n^\#:D_n^\#\to Y^\#\) is surjective for all $n\ge2$, and the induced map \(q^\#:D^\#\to Y^\#\) is surjective as well, where
\[ q =\colim_{n\in\NN} q_n: D =\colim_{n\in\NN} D_n \to Y.
\]
Diagram~\eqref{dia-XDDD-Y} induces also diagram of cones
\[ \Cone q_0^\# \rTTo^{\Cone(h_0^\#,1)} \Cone q_1^\# \rTTo^{\Cone(h_1^\#,1)} \Cone q_2^\# \to \dots \to \Cone q^\# =\colim_{n\in\NN} \Cone q_n^\#.
\]
It follows from \propref{pro-N-dg-submodule-cycles-Cone} that the submodule of cycles \(Z\Cone q_n^\#\) is taken by \(\Cone(h_n^\#,1)\) to the submodule of boundaries \(B\Cone q_{n+1}^\#\).
Thus the colimit of cones \(\Cone q^\#\) is acyclic.
Therefore, $q^\#$ is a quasi-isomorphism.
We have decomposed a morphism $f\in\cc$ into a standard cofibration $i$ and a trivial fibration $q$: \(f=\bigl(X\rto i D\rto q Y\bigr)\).

\subsubsection*{\textup{(MC4)(ii)}.} 
Let us prove that a standard trivial cofibration \(\bar\jmath:X\to X\langle M,0\rangle\) is in \(\sS{^\perp}\cR_f\).
Here \(M[1]=\Cone1_{N[-1]}\) and $N$ consists of free $\kk$\n-modules.
We have \(X\langle M,0\rangle= F(M[1])\sqcup X\).
A square
\begin{diagram}
X &\rTTo^a &A
\\
\dTTo<{\bar\jmath} &\ruDashTo^c &\dFib>g
\\
X\langle M,0\rangle &\rTTo^b &B
\end{diagram}
commutes iff \(b=l\cup ag:F(M[1])\sqcup X\to B\).
The adjunction takes $l$ to \(l^t:M[1]\to B^\#\in\dg^S\).
There is a commutative diagram in $\Set$
\begin{diagram}[LaTeXeqno]
\dg^S(M[1],A^\#) &\rTTo^\sim &\udg^S(N,A^\#)^0
\\
\dTTo<{\dg^S(1,g^\#)} &&\dTTo>{\udg^S(1,g^\#)}
\\
\dg^S(M[1],B^\#) &\rTTo^\sim &\udg^S(N,B^\#)^0
\label{dia-dg-udg-dg-udg}
\end{diagram}

Assume that $g\in\cR_f$, that is, $g^\#$ is surjective in each degree.
Since $N$ consists of free $\kk$\n-modules, the vertical maps are surjections.
Thus, there is a chain map \(r:M[1]\to A^\#\) such that \(l^t=r\cdot g^\#\).
Using adjunction we find that \(l=\bigl(F(M[1])\rto{\sS{^t}r} A\rto{g} B\bigr)\).
Then \(c=\sS{^t}r\cup a:F(M[1])\sqcup X\to A\) is the sought diagonal filler.

Denote by $J$ the class of all standard trivial cofibrations.
Then the above reasoning turned backward shows that for $g\in J^\perp$ vertical arrows of \eqref{dia-dg-udg-dg-udg} are always surjective which implies that $g\in\cR_f$.
Hence, $J^\perp=\cR_f$.

Consider an arbitrary morphism \(f:X\to Y\in\cc\).
Accordingly to proven (MC5)(ii) there is a decomposition
\begin{diagram}
X &\rTTo^{\bar\jmath} &Z=X\langle M,0\rangle
\\
\dTTo<f &&\dFib>p
\\
Y &\rEq &Y
\end{diagram}
into a standard trivial cofibration $\bar\jmath$ and \(p\in\cR_f\).
If \(f\in\cw\cap\cL_c\), then \(p\in\cw\cap\cR_f\).
By definition of \(\cL_c=\sS{^\perp}(\cW\cap\cR_f)\) there is a morphism $w$ such that the following diagrams commute
\begin{equation}
\begin{diagram}[inline]
X &\rTTo^{\bar\jmath} &Z
\\
\dTTo<f &\ruTTo^w &\dTTo>p
\\
Y &\rEq &Y
\end{diagram}
\quad \Longleftrightarrow \quad
\begin{diagram}[inline]
X &\rEq &X &\rEq &X
\\
\dTTo<f &&\dTTo<{\bar\jmath} &&\dTTo>f
\\
Y &\rTTo^w &Z &\rTTo^p &Y
\end{diagram}
\label{eq-dia-fjf-wp1}
\end{equation}
and \(w\cdot p=1_Y\).
That is, $f$ is a retract of $\bar\jmath$.
Hence, \((\cW\cap\cR_f)^\perp=J^\perp=\cR_f\).
\end{proof}

\begin{remark}
It is shown in the proof that any trivial cofibration $f$ is a retract of a standard trivial cofibration $\bar\jmath$ of type \eqref{eq-dia-fjf-wp1}, \textit{cf.} \cite[Remark~2.2.5]{Hinich:q-alg/9702015}.
Similarly, any cofibration $f$ is a retract of a standard cofibration $\bar\jmath$ of type \eqref{eq-dia-fjf-wp1}.
The model structure of $\cc$ is cofibrantly generated by the classes of elementary cofibrations and of standard trivial cofibrations.
\end{remark}


\begin{thebibliography}{Mur11}

\bibitem[BM03]{math.AT/0206094}
Clemens Berger and Ieke Moerdijk, \emph{Axiomatic homotopy theory for operads},
  Comment. Math. Helv. \textbf{78} (2003), no.~4, 805--831,
  \href{http://arXiv.org/abs/math.AT/0206094}{{\tt
  arXiv:\linebreak[1]math.AT/0206094}}.

\bibitem[CH02]{math.KT/0011216}
J.~Daniel Christensen and Mark Hovey, \emph{Quillen model structures for
  relative homological algebra}, Math. Proc. Cambridge Philos. Soc.
  \textbf{133} (2002), no.~2, 261--293,
  \href{http://arXiv.org/abs/math.KT/0011216}{{\tt
  arXiv:\linebreak[1]math.KT/0011216}}.

\bibitem[Har10]{MR2593672}
John~Edward Harper, \emph{Homotopy theory of modules over operads and
  non-{$\Sigma$} operads in monoidal model categories}, J. Pure Appl. Algebra
  \textbf{214} (2010), no.~8, 1407--1434,
  \href{http://arXiv.org/abs/0801.0191}{{\tt arXiv:\linebreak[1]0801.0191}}.

\bibitem[Hin97]{Hinich:q-alg/9702015}
Vladimir Hinich, \emph{Homological algebra of homotopy algebras}, Comm. Algebra
  \textbf{25} (1997), no.~10, 3291--3323,
  \href{http://arXiv.org/abs/q-alg/9702015}{{\tt
  arXiv:\linebreak[1]q-alg/9702015}}.

\bibitem[Lyu11]{Lyu-Ainf-Operad}
Volodymyr Lyubashenko, \emph{Homotopy unital ${A}_\infty$-algebras}, J. Algebra \textbf{329}
  (2011), no.~1, 190--212,
  \url{http://dx.doi.org/10.1016/j.jalgebra.2010.02.009} Special Issue
  Celebrating the 60th Birthday of Corrado De Concini
  \href{http://arXiv.org/abs/1205.6058}{{\tt arXiv:\linebreak[1]1205.6058}}.

\bibitem[Lyu12]{Lyu-A8-several-entries}
Volodymyr Lyubashenko, \emph{${A}_\infty$-morphisms with several entries},
 may 2012, \href{http://arXiv.org/abs/1205.6072}{{\tt arXiv:\linebreak[1]1205.6072}}.

\bibitem[Mur11]{MR2821434}
Fernando Muro, \emph{Homotopy theory of nonsymmetric operads}, Algebr. Geom.
  Topol. \textbf{11} (2011), no.~3, 1541--1599,
  \href{http://arXiv.org/abs/1101.1634}{{\tt arXiv:\linebreak[1]1101.1634}}
  \url{http://dx.doi.org/10.2140/agt.2011.11.1541}.

\bibitem[Spi01]{math/0101102}
Markus Spitzweck, \emph{Operads, algebras and modules in general model
  categories}, jan 2001, \href{http://arXiv.org/abs/math/0101102}{{\tt
  arXiv:\linebreak[1]math/0101102}}.

\bibitem[SS00]{math.AT/9801082}
Stefan Schwede and Brooke Shipley, \emph{Algebras and modules in monoidal model
  categories}, Proc. London Math. Soc. (3) \textbf{80} (2000), no.~2, 491--511,
  \href{http://arXiv.org/abs/math.AT/9801082}{{\tt
  arXiv:\linebreak[1]math.AT/9801082}}.

\end{thebibliography}
\ifx\chooseClass1
	\else
\tableofcontents
\fi

\end{document}